\documentclass[a4paper, 10pt, parskip=half]{scrartcl}

\usepackage[utf8]{inputenc}
\usepackage[T1]{fontenc}
\usepackage{lmodern}
\usepackage{csquotes}
\usepackage{amsmath}
\usepackage{amssymb}
\usepackage{amsthm}
\usepackage{amsfonts}
\usepackage{dsfont}
\usepackage{mathtools}
\usepackage{marginnote}
\usepackage[dvipsnames]{xcolor}
\usepackage{hyperref}
\hypersetup{%
  colorlinks=true,
  linkcolor=black,
  citecolor=black,
  urlcolor=blue,
  pdftitle={Critical mass phenomena in higher dimensional quasilinear Keller--Segel systems with indirect signal production},
  pdfauthor={},
  pdfkeywords={},
  bookmarksopen=true,
}

\usepackage{authblk}
\usepackage{float}
\usepackage{subcaption}
\usepackage{esint}

\usepackage[numbers]{natbib}
\newcommand{\R}{\mathbb{R}}

\newcommand{\mc}[1]{\mathcal{#1}}
\newcommand{\ur}[1]{\mathrm{#1}}
\newcommand{\ure}{\ur{e}}

\ifdefined\labelenumi%
  \renewcommand{\labelenumi}{(\roman{enumi})}
  
\fi

\newcommand{\eps}{\varepsilon}

\newcommand{\defs}{\coloneqq}
\newcommand{\sfed}{\eqqcolon}

\newcommand{\ol}{\overline}
\newcommand{\ul}{\underline}
\newcommand{\wt}{\widetilde}

\newcommand{\ds}{\,\mathrm{d}s}

\newcommand{\dr}{\,\mathrm{d}r}

\newcommand{\ddt}{\frac{\mathrm{d}}{\mathrm{d}t}}

\newcommand{\tmax}{T_{\max}}

\newcommand{\intom}{\int_\Omega}

\newcommand{\Ombar}{\ol \Omega}

\newcommand{\con}[2][\Ombar]{\ensuremath{C^{#2}(#1)}}

\newcommand{\f}[2]{\frac{#1}{#2}}

\newcommand{\set}[1]{\left\{#1\right\}}

\newcommand{\kl}[1]{\left(#1\right)}
\newcommand{\io}{\int_\Omega}

\usepackage{newunicodechar}
\newunicodechar{∇}{\nabla}
\newunicodechar{Δ}{\Delta}
\newunicodechar{α}{\alpha}
\newunicodechar{γ}{\gamma}
\newunicodechar{μ}{\mu}
\newunicodechar{ℕ}{\mathbb{N}}
\newunicodechar{ℝ}{\mathbb{R}}
\newunicodechar{∂}{\partial}
\newunicodechar{ν}{\nu}
\newunicodechar{Φ}{\Phi}
\newunicodechar{τ}{\tau}
\newunicodechar{ω}{\omega}
\newunicodechar{∞}{\infty}
\newunicodechar{ϕ}{\phi}
\newunicodechar{χ}{\chi}
\newunicodechar{ξ}{\xi}
\newunicodechar{ε}{\varepsilon}

\makeatletter
\renewenvironment{proof}[1][\proofname]{\par
  \pushQED{\qed}%
  \normalfont \topsep0\p@\relax
  \trivlist
  \item[\hskip\labelsep\scshape
  #1\@addpunct{.}]\ignorespaces
}{%
  \popQED\endtrivlist\@endpefalse
}
\makeatother

\newtheorem{base}{Base}[section]
\numberwithin{equation}{section}

\newtheorem{theorem}[base]{Theorem} \newtheorem*{theorem*}{Theorem}
\newtheorem{lemma}[base]{Lemma} \newtheorem*{lemma*}{Lemma}
\newtheorem{prop}[base]{Proposition} \newtheorem*{prop*}{Proposition}
\newtheorem{cor}[base]{Corollary} \newtheorem*{cor*}{Corollary}
 \newtheorem*{algo*}{Algorithm}

\theoremstyle{definition}
\newtheorem{remark}[base]{Remark} \newtheorem*{remark*}{Remark}
 \newtheorem*{definition*}{Definition}
 \newtheorem*{example*}{Example}
 \newtheorem*{cond*}{Condition}

\allowdisplaybreaks

\setkomafont{title}{\normalfont\Large}
\title{Critical mass phenomena in higher dimensional quasilinear Keller--Segel systems with indirect signal production}
\usepackage{authblk}

\author[1]{Mario Fuest\footnote{e-mail: fuest@ifam.uni-hannover.de}}
\author[1]{Johannes Lankeit\footnote{e-mail: lankeit@ifam.uni-hannover.de}}
\author[2]{Yuya Tanaka\footnote{e-mail: yuya.tns.6308@gmail.com}}
\affil[1]{Leibniz Universität Hannover, Institut für Angewandte Mathematik, Welfengarten 1, 30167 Hannover, Germany}
\affil[2]{Tokyo University of Science, Department of Mathematics, 1-3, Kagurazaka, Shinjuku-ku, Tokyo 162-8601, Japan}

\date{}

\begin{document}

\maketitle

\KOMAoptions{abstract=true}
\begin{abstract}
\noindent
In this paper, we deal with quasilinear Keller--Segel systems
with indirect signal production, 
  \[
    \begin{cases}
      u_t = \nabla \cdot ((u+1)^{m-1}\nabla u)  - \nabla \cdot (u \nabla v), &x \in \Omega,\ t> 0,\\
      0 = \Delta v - \mu(t) + w, &x \in \Omega,\ t> 0,\\
      w_t + w = u, &x \in \Omega,\ t> 0,
    \end{cases}
  \]
complemented with homogeneous Neumann boundary conditions and suitable initial conditions,
where $\Omega\subset\mathbb R^n$ $(n\ge3)$ is a bounded smooth domain, $m\ge1$ and
  \[
     \mu(t)\defs \fint_\Omega w(\cdot, t) \qquad\mbox{for}\ t>0.
  \]
We show that in the case $m\ge2-\frac{2}{n}$, 
there exists $M_c>0$ 
such that if either $m>2-\frac{2}{n}$ 
or $\int_\Omega u_0 <M_c$, then 
the solution exists globally and remains bounded, 
and that in the case $m\le2-\frac{2}{n}$, if either $m<2-\frac{2}{n}$ or 
$M>2^\frac{n}{2}n^{n-1}\omega_n$, 
then there exist radially symmetric initial data such that 
$\int_\Omega u_0 = M$ and the solution blows up 
in finite or infinite time, where the blow-up time is infinite if $m=2-\frac2n$.\\[0.5pt]
In particular, if $m=2-\frac{2}{n}$ there is a critical mass phenomenon in the sense that
\[\inf\left\{M > 0 : \exists u_0 \text{ with } \int_\Omega u_0 = M \text{ such that the corresponding solution blows up in infinite time}\right\}\]
is a finite positive number.
  \\[0.5pt]
 \textbf{AMS Classification (2020):} {
35B33 
(primary);
35B44 
35A01, 
35K55, 
92C17 
(secondary)
}
\end{abstract}

\section{Introduction}

\paragraph{Critical mass in the two-dimensional Keller--Segel system with direct signal production}

Chemotaxis, which is the motion of cells oriented towards higher concentrations of a chemical substance, is an important cause for aggregation in different biological contexts, e.g.\ the formation of bacterial colonies or tumour invasion \cite{survey_hillen_painter}. For its mathematical description, the Keller--Segel model is often used, which, in a simplified parabolic--elliptic form reads 
  \begin{equation}\label{KS}
    \begin{cases}
      u_t = \Delta u  - \nabla \cdot (u \nabla v),\\
      0 = \Delta v - v + u.
    \end{cases}
  \end{equation}
Here, $u$ and $v$ denote the density of cells and the concentration of a signal 
substance, respectively, and the chemical signal is directly produced by the cells. 
This system and its variants have been studied extensively (see e.g.\ the surveys \cite{H_2003_JDMV,BBTW,LW_2020_JDMV}). 

For the present article, it is of particular interest that \eqref{KS} features a critical-mass phenomenon in that in $2$-dimensional domains,  radially symmetric initial data $u_0$ with mass $M=\io u_0<8\pi$  lead to global and bounded solutions, whereas for any larger mass, some initial data with this mass can be found which evolve into solutions blowing up in finite time, \cite{Nagai_1995}. In higher-dimensional settings, blow-up solutions can be found for any prescribed positive initial mass, \cite{Nagai_1995}.
In the radially symmetric case, the same has been observed for other parabolic--elliptic and fully parabolic versions of \eqref{KS}, 
see \cite{HV,MW_pre,NSY,Senba_2005,W_2013_JMPA}.

\paragraph{Indirect signal production}
A more recent line of investigations is concerned with systems where the signal is produced indirectly, for example as in 
  \begin{align}\label{MPB}
    \begin{cases}
      u_t = \Delta u  - \nabla \cdot (u \nabla v), &x \in \Omega,\ t> 0,\\
      0 = \Delta v - \mu(t) + w, &x \in \Omega,\ t> 0,\\
      w_t + w = u, &x \in \Omega,\ t> 0,\\
      \nabla u \cdot \nu = \nabla v \cdot \nu = 0, &x \in \partial\Omega,\ t> 0,\\
      u(x,0) = u_0(x), \quad w(x,0) = w_0(x), &x \in \Omega.
    \end{cases}
  \end{align}

Here, $\Omega\defs B_1(0) \subset \R^n$ $(n\ge2)$ is a ball and 
  \[ 
    \mu(t)\defs \fint_\Omega w \quad\mbox{for}\ t>0; 
  \] 
$\nu$ is the outward normal vector to $\partial \Omega$;
$u_0 \in C^0(\Ombar)$ and $w_0 \in C^1(\Ombar)$ 
are nonnegative.  
This is the simplified version of 
a chemotaxis model proposed by Strohm, Tyson and Powell \cite{STP}, 
which describes the spread and aggregative behaviour of 
the Mountain Pine Beetle (MPB). 
Here, $u,w$ represent the densities of flying MPB and of nesting MPB, 
and $v$ denotes the concentration of the beetle pheromone. 
For the above model, Tao and  Winkler \cite{Tao-W_2017_JEMS} 
established boundedness and infinite-time blow-up 
in the two-dimensional case;
more precisely, in the radial setting, 
the solution remains bounded when $\int_\Omega u_0 < 8\pi$, 
and there exist initial data such that $\int_\Omega u_0 > 8\pi$ 
and the solution blows up in infinite time. 
Based on a Lyapunov functional, Lauren\c{c}ot \cite{L_2019_DCDSB} uncovered the same phenomenon in the fully parabolic and in the nonradial setting, where the critical mass decreases to $4\pi$.
Also in a related model concerned with a population split into a static, signal-producing and a motile, chemotactically active group, a dichotomy between initial masses leading to global boundedness of all solutions or unboundedness, respectively, has been observed \cite{LS_2021_SIAM}.
From these results, we understand that 
a critical mass phenomenon happens also 
for the two-dimensional Keller--Segel system 
with indirect signal production given above and 
that the difference from systems with direct signal production is that 
solutions always exist globally in time.

In summary, also for indirect signal production, there is still some critical mass phenomenon in 2D, but now it discriminates between boundedness and unboundedness of global solutions. 

However, up to now, it is not known whether such results are satisfied in the higher 
dimensional cases. 
Thus the following question naturally arises:

\begin{center}
\textit{Does a critical mass phenomenon happen 
in the higher dimensional Keller--Segel system\\
with indirect signal production?}
\end{center}

%

As to be shown below, the answer is no. 
In fact, we always obtain unbounded solutions
in the higher dimensional cases.

\begin{prop}\label{prop1.1}
Let $\Omega\defs B_1(0) \subset \mathbb{R}^n$ $(n\ge3)$ be a ball. 
For each $M > 0$, there exist initial data $(u_0, w_0)$ with $\intom u_0 = M$ such that the corresponding solution of 
\eqref{MPB} is unbounded.
\end{prop}
We remark that this is in contrast to the indirect taxis system considered in \cite{FujieSenbaApplicationAdamsType2017, FujieSenbaBlowupSolutionsTwochemical2019},
where also the third component diffuses and where the mass $(8 \pi)^2$ has been observed to be critical in dimension four.

\paragraph{Critical mass phenomenon in higher dimensional \emph{quasilinear} Keller--Segel system with indirect signal production}

Quasilinear diffusion is occasionally relevant in applications of chemotaxis systems, e.g.\ when concerned with the motion of cells (like tumour invasion, cf.~\cite{LiLankeit}%
).
In the present indirect setting, the inclusion of porous medium type diffusion leads to the following system 
  \begin{align}\label{P}
    \begin{cases}
      u_t = \nabla \cdot ((u+1)^{m-1}\nabla u)  - \nabla \cdot (u \nabla v), &x \in \Omega,\ t> 0,\\
      0 = \Delta v - \mu(t) + w, &x \in \Omega,\ t> 0,\\
      w_t + w = u, &x \in \Omega,\ t> 0,\\
      \nabla u \cdot \nu = \nabla v \cdot \nu = 0, &x \in \partial\Omega,\ t> 0,\\
      u(x,0) = u_0(x), \quad w(x,0) = w_0(x), &x \in \Omega,
    \end{cases}
  \end{align}
where $\Omega\subset\R^n$ $(n\ge3)$ is a bounded smooth domain and $\mu(t)\defs \fint_\Omega w$ for $t>0$; 
$m\ge1$; 
$\nu$ is the outward normal vector to $\partial \Omega$; 
$u_0 \in C^0(\Ombar)$ and $w_0 \in C^1(\Ombar)$ 
are nonnegative. 

For a related direct production system 
  \[
    \begin{cases}
      u_t = \nabla \cdot ((u+1)^{m-1}\nabla u)  - \nabla \cdot (u \nabla v),\\
      v_t = \Delta v - v + u,
    \end{cases}
  \]
it is known that the size of $m$ 
determines whether solutions remain bounded or blow up; large diffusion exponents $m$ counteract explosions; 
indeed, when $m>2-\frac{2}{n}$, boundedness of solutions was 
obtained in \cite{ISY, Tao-W_2012_JDE}; 
on the other hand, when $m<2-\frac{2}{n}$, 
initial data with corresponding solutions blowing up in finite time were constructed
in \cite{CS_2012,CS_2015}. 
Also, for the system such that the diffusion term is replaced 
with the degenerate diffusion $\Delta u^m$, 
similar results were proved 
(boundedness in \cite{IY_2012_JDE, Mimura_2017} and 
finite-time blow-up in \cite{HIY}). 
Moreover, in the critical case $m=2-\frac{2}{n}$, 
a critical mass phenomenon is observed in the degenerate system; 
in this case, there exists a critical mass $M_c>0$ such that 
if $\int_\Omega u_0 < M_c$, then solutions 
are global and bounded (\cite{BL_2013_CPDE, Mimura_2017}), 
whereas if $\int_\Omega u_0 > M_c$, then 
there exist initial data leading to finite-time blow-up
in the three- or four-dimensional cases (\cite{LM_2017}).

From these results, 
the questions naturally arise 
whether behaviour of solutions is determined by conditions on $m$ 
also in \eqref{P} and whether a critical mass phenomenon 
possibly happens for some value of $m$.
The purpose of this paper is to give 
a positive answer to the aforementioned question 
in the system \eqref{P}.

\paragraph{Main results}
As preparation -- and in order to specify what type of solutions and possible blow-up we are dealing with -- we introduce the following proposition 
on local existence of solutions to \eqref{P}.

\begin{prop}[Local existence]\label{local}
Let $\Omega \subset \R^n$ $(n\ge3)$ 
be a bounded smooth domain and let $m\ge1$. 
Assume that $u_0 \in C^0(\Ombar)$ and $w_0 \in C^1(\Ombar)$ 
are nonnegative. 
Then there exist $\tmax \in (0,\infty]$ 
and uniquely determined nonnegative functions 
  \begin{align*}
    u &\in C^0(\Ombar \times [0,\tmax)) \cap C^{2,1}(\Ombar \times (0,\tmax)),\\
    v &\in  C^{2,0}(\Ombar \times [0,\tmax)), 
\\
    w &\in C^{0,1}(\Ombar \times [0,\tmax))
          \defs \{w\in C^0(\Ombar\times[0,\tmax))\mid w_t\in C^0(\Ombar\times[0,\tmax))\}, 
  \end{align*}
which solve \eqref{P} classically in $\Ombar \times [0,\tmax)$ 
and which are such that 
  \begin{align}\label{criterion}
    \mbox{if}\ \tmax<\infty, \quad\mbox{then}\ 
    \| u(\cdot,t) \|_{L^\infty(\Omega)} \to \infty 
    \quad\mbox{as}\ t \nearrow \tmax.
  \end{align}
Moreover, if $u_0$ and $w_0$ are radially symmetric, 
then so are $u(\cdot,t),v(\cdot,t)$ and $w(\cdot,t)$ for any $t \in (0,\tmax)$.
\end{prop}

The first theorem is concerned with global existence and boundedness
of solutions to \eqref{P}. 
Similarly as in 
\cite{BL_2013_CPDE, ISY, IY_2012_JDE, Mimura_2017, Tao-W_2012_JDE}, 
solutions remain bounded when $m>2-\frac{2}{n}$, 
and the same result holds under a smallness condition for the size of initial data $u_0$ when $m=2-\frac{2}{n}$.
Moreover, we note that global existence is ensured 
regardless of the size of initial data when $m\ge2-\frac{2}{n}$.

\begin{theorem}[Global existence and boundedness]\label{GB}  
Let $\Omega \subset \R^n$ $(n\ge3)$ be a bounded smooth domain 
and let $m\ge2-\frac{2}{n}$. 
 Then there exists $M_c>0$ satisfying the following property:
For all nonnegative initial data 
$u_0 \in C^0(\Ombar)$ and $w_0 \in C^1(\Ombar)$ 
the corresponding solution $(u,v,w)$ of \eqref{P} exists globally in time. 
Moreover, if $u_0$ and $w_0$ satisfy either

  \[
    m>2-\frac{2}{n}
    \quad\mbox{or}\quad
    \int_\Omega u_0<M_c,
  \]
then the solution $(u,v,w)$ of \eqref{P} is bounded 
in $\Omega\times(0,\infty)$ in the sense that there is $C>0$ such that
  \begin{align}\label{bdd}
    u(x,t) \le C, \quad v(x,t) \le C 
    \quad\mbox{and}\quad 
    w(x,t) \le C
  \end{align}
for all $x\in\Omega$ and $t\in(0,\infty)$.
\end{theorem}

The next theorem gives unboundedness of solutions to \eqref{P}. 

\begin{theorem}[Unboundedness]\label{BU} 
Let $\Omega\defs B_1\defs B_1(0) \subset \mathbb{R}^n$ $(n\ge3)$ be a ball 
and let $m\in\big[1,2-\frac{2}{n}\big]$. 
If either 
  \[
    \left( m\in\left[1,2-\frac{2}{n}\right) \text{ and } M>0\right)
    \quad\mbox{or}\quad
    \left(m=2-\f2n \text{ and } M>2^\frac{n}{2}n^{n-1}\omega_n\right), 
  \]
where $\omega_n\defs |\partial B_1(0)|$, 
then there exist $R\in(0,1)$ and $\alpha>0$ such that 
for each $\eta>0$ one can find constants
$\Gamma_u>0$, $\gamma>0$ and $\Gamma_w>0$
with the property that for all nonnegative radially symmetric initial data 
$u_0\in C^0(\Ombar)$ and $w_0\in C^1(\Ombar)$ satisfying
  \begin{alignat}{2}
    &\int_\Omega u_0 = M, 
  \notag \\
    \label{BUcondi2}
    &\fint_{B_r} u_0 \ge \Gamma_u
    &\quad &\mbox{for all}\ r\in(0,R),
  \\ 
    \label{BUcondi3}
    &\fint_{B_1\setminus{B_r}} u_0 \le \gamma
    & &\mbox{for all}\ r\in(R,1),
  \\ 
    \label{BUcondi4}
    &\fint_{B_r} w_0 \ge \fint_{B_1} w_0 + \Gamma_w
    & &\mbox{for all}\ r\in(0,R), 
  \\ 
    \label{BUcondi5}
    &\fint_{B_1\setminus{B_r}} w_0 \le \fint_{B_1} w_0 - \eta
    & &\mbox{for all}\ r\in(R,1),
  \end{alignat}
the solution $(u,v,w)$ of \eqref{P} 
blows up in finite or infinite time 
in the sense that $\set{\|u(\cdot,t)\|_{L^\infty(\Omega)}\mid t\in (0,\tmax)}$ is unbounded.
More precisely, there exist $\alpha > 0$ and $C > 0$ such that $\|u(\cdot,t)\|_{L^\infty(\Omega)} \ge C \ure^{\alpha t}$ for all $t \in (0, \tmax)$.
\end{theorem}

From this theorem, we can immediately prove Proposition~\ref{prop1.1}. 

\begin{proof}[Proof of Proposition 1.1]
Let $m=1$ and $M>0$. Then, from Theorem~\ref{BU}, we can 
find initial data $u_0,w_0$ with $\int_\Omega u_0=M$ 
such that the corresponding solution is unbounded.
\end{proof}%

In Theorem~\ref{BU}, we see that 
blow-up occurs for the system \eqref{P} 
in the case $m<2-\frac{2}{n}$ 
as in \cite{CS_2012, CS_2015, HIY, LM_2017}.
However, it is not clear whether the blow-up time is finite or infinite.
In contrast, in the case $m=2-\frac{2}{n}$ 
we can obtain initial data leading to infinite-time blow-up 
by a combination of Theorems~\ref{GB} and \ref{BU}.

\begin{cor}\label{cor}
Let $\Omega\defs B_1(0) \subset \mathbb{R}^n$ $(n\ge3)$ be a ball 
and let $m=2-\frac{2}{n}$. 
Then for all 
  \[
    M>2^\frac{n}{2}n^{n-1}\omega_n,
  \]
there exist nonnegative radially symmetric initial data $(u_0,w_0)$ such that $\int_\Omega u_0 = M$ and the solution blows up in infinite time.
\end{cor}

\begin{remark}
From Theorem~\ref{GB} and Corollary~\ref{cor}, 
we know that the critical value is $m=2-\frac{2}{n}$ 
in the higher dimensional cases and 
that behaviour of solutions is divided by the size of initial data.
Here, we note that this is different from results on
Keller--Segel systems with direct signal production (see, e.g., \cite{CS_2015}) in 
that the blow-up time is always infinite.
\end{remark} 

\textbf{Main ideas and plan of the paper.} 
 In Section~\ref{sec3}, we will prove Theorem~\ref{GB}. 
Both global existence and boundedness of solutions are obtained by showing that the functional
\begin{align*}
  \frac1p \intom u^p + \frac{1}{p+1} \intom w^{p+1}
\end{align*}
(which we consider for sufficiently large $p$) is a subsolution to a certain ODE, whose solution is global and, under certain conditions, also bounded.
A crucial step for the corresponding estimates is the observation that the third equation in \eqref{P} regularizes in time (but not in space),
which is inter alia manifested by the identity $\frac{1}{p+1} \ddt\intom w^{p+1} + \intom w^{p+1}=\intom uw^p$.
If $m \ge 2-\frac2n$, the Gagliardo--Nirenberg inequality allows us then to favorably bound the worrisome terms by, essentially, the dissipative terms;
if either $m > 2-\frac2n$ or $\intom u_0$ is sufficiently small, this can be done 
(see Lemmata~\ref{lem3.3} and \ref{lem3.4}) 
in such a way that we can not only infer global existence but also boundedness of the solutions. 
This proof of global existence is based on
\cite[Section 3]{Tao-W_2017_JEMS}.

Section~\ref{sec4} is devoted to proving Theorem~\ref{BU}.
To that end, we consider the function $U(\xi,t) \defs  \int^{\xi^\frac{1}{n}}_0 r^{n-1}u(r,t)\dr$
for $\xi\in[0,1]$ and $t\in[0,\tmax)$,
which is introduced in \cite{JL}. 
As in \cite{Tao-W_2017_JEMS}, our aim is to construct an unbounded subsolution $\ul U$.
Here, we take the ansatz
$\ul{U}(\xi,t)\defs \frac{a(t)\xi}{b(t)+\xi}$ for $\xi\in[0,\xi_0]$ and 
$t\in[0,\infty)$ with certain $a, b$ and a small $\xi_0\in(0,1)$, 
so that the derivative of $\ul U$ is unbounded at the origin.
One of the new challenges compared to \cite{Tao-W_2017_JEMS} is to deal with the term 
  \[
    2n^2 \kl{\frac{na(t)b(t)}{(b(t)+\xi)^2}+1}^{m-1}\frac{\xi^{1-\frac{2}{n}}}{b(t)+\xi},
  \]
which appears when applying the parabolic operator $\mc P$ defined in \eqref{defP} to $\ul U$.
Evidently, this term is simplified if one sets $m=1$.
However, if $m\in\big[1,2-\frac{2}{n}\big]$, this term can still be controlled by 
$\xi_0^{2-\frac{2}{n}-m}$, and next absorbed by 
a negative term depending on $\intom u_0$; 
if either $m\in\big[1,2-\frac{2}{n}\big)$ or $\intom u_0$ is sufficiently large, 
this approach works and thus we can indeed show that $\ul U$ and hence $U$ is unbounded.

\section{Preliminaries}\label{sec2}
In this section, let $\Omega \subset \mathbb{R}^n$ 
$(n\ge3)$ be a bounded domain with a smooth boundary and $m\ge1$.
To begin with, we consider the local existence result of Proposition~\ref{local}.

\begin{proof}[Proof of Proposition~\ref{local}]
This proposition is proved by a standard fixed point argument as in \cite{Tao-W_2011_SIAM}.
\end{proof}

We next collect a basic property of solutions to \eqref{P}.

\begin{lemma}
Let $m\ge1$ and let $u_0\in C^0(\Ombar)$ and $w_0\in C^1(\Ombar)$ be nonnegative. 
Then the solution $(u,v,w)$ of \eqref{P} given by Proposition~\ref{local} with maximal existence time $\tmax$ satisfies that
  \begin{align}\label{mass}
    \int_\Omega u(\cdot,t) = \int_\Omega u_0
  \end{align}
  and
  \begin{align}\label{mu}
    \mu(t)=\fint_\Omega w(\cdot,t)\ge0 
    \qquad 
  \end{align}
  for all $t \in (0, \tmax)$.
\end{lemma}
\begin{proof}
  While \eqref{mass} follows immediately upon integrating the first equation in \eqref{P}, \eqref{mu} is obtained by nonnegativity of $w$.
\end{proof}

\section{Global existence and boundedness}\label{sec3}
Throughout this section, we fix a smooth, bounded domain $\Omega \subset \mathbb{R}^n$ 
$(n\ge3)$, parameters $m\ge1$ and $M>0$ as well as nonnegative initial data $u_0\in C^0(\Ombar)$, $w_0 \in C^1(\Ombar)$ with $M=\int_\Omega u_0$.
Moreover, we denote the solution of \eqref{P} provided by Proposition~\ref{local} by $(u,v,w)$ and its maximal existence time by $\tmax$.

In order to prove global existence and boundedness, 
we will establish $L^p$-estimates for $u$ and $L^{p+1}$-estimates for $w$. 
To this end, we begin by preparing an estimate for 
$\frac{1}{p} \int_\Omega u^p + \frac{1}{p+1} \int_\Omega w^{p+1}$.

\begin{lemma}
 For all $p>1$ and $k>0$, the solution $(u, v, w)$ of \eqref{P} satisfies 
  \begin{align}\label{lem3.1;ineq}
    &\ddt \set{\frac{1}{p} \int_\Omega u^p + \frac{1}{p+1} \int_\Omega w^{p+1}} 
    + \frac{4(p-1)}{(p+m-1)^2} \int_\Omega |\nabla u^\frac{p+m-1}{2}|^2 
    + \int_\Omega w^{p+1}
    \notag \\
    &\quad\,\le 
    2  k \int_\Omega u^{p+1}
    +  (k^{-p} + k^{-\frac{1}{p}}) \int_\Omega w^{p+1}
    \qquad \text{in $(0, \tmax)$}.
  \end{align}
\end{lemma}

\begin{proof} 
For henceforth fixed $p>1$, we multiply the first equation in \eqref{P} by $u^{p-1}$, integrate over $\Omega$ and use integration by parts to obtain that
  \begin{align*}
    \frac{1}{p} \ddt \int_\Omega u^p 
    &= \int_\Omega u^{p-1} \nabla \cdot ((u+1)^{m-1}\nabla u)  
       - \int_\Omega u^{p-1} \nabla \cdot (u \nabla v)
    \\
    &= - (p-1) \int_\Omega u^{p-2} (u+1)^{m-1} |\nabla u|^2 
         + (p-1) \int_\Omega u^{p-1} \nabla u \cdot \nabla v
    \\ 
    &\le -\frac{4(p-1)}{(p+m-1)^2} \int_\Omega |\nabla u^\frac{p+m-1}{2}|^2 
           - \frac{p-1}{p} \int_\Omega u^p \Delta v
    \qquad \text{in $(0,\tmax)$.}
  \end{align*}
Here, from the second equation in \eqref{P} and \eqref{mu}
we have $\Delta v = \mu(t)-w \ge -w$, 
so that it follows that 
  \begin{align}\label{lem3.1;ineq1}
    \frac{1}{p} \ddt \int_\Omega u^p 
    \le 
    -\frac{4(p-1)}{(p+m-1)^2} \int_\Omega |\nabla u^\frac{p+m-1}{2}|^2 
    + \frac{p-1}{p} \int_\Omega u^p w
    \qquad \text{in $(0,\tmax)$.}
  \end{align}
Next, multiplying the third equation in \eqref{P} by $w^p$ and 
integrating over $\Omega$, we can observe that 
  \[
    \frac{1}{p+1} \ddt \int_\Omega w^{p+1} + \int_\Omega w^{p+1} 
    = \int_\Omega uw^p, 
  \]
which together with \eqref{lem3.1;ineq1} implies that 
  \[
    \ddt \set{\frac{1}{p} \int_\Omega u^p + \frac{1}{p+1} \int_\Omega w^{p+1}}
    + \frac{4(p-1)}{(p+m-1)^2} \int_\Omega |\nabla u^\frac{p+m-1}{2}|^2 
    + \int_\Omega w^{p+1} 
    \le
    \frac{p-1}{p} \int_\Omega u^p w
    + \int_\Omega uw^p
  \]
in $(0,\tmax)$. 
Finally, we apply Young's inequality to two terms 
on the right-hand side of this inequality, and thereby 
derive \eqref{lem3.1;ineq}.
\end{proof}

To estimate the first term on the right-hand side of \eqref{lem3.1;ineq}, 
we next state the following consequence of the Gagliardo--Nirenberg inequality.

\begin{lemma}\label{GNineq}
 For all $p > \max\set{1,\frac{n}{2}\kl{2-\frac{2}{n}-m}}$,
there exists $C>0$ such that for arbitrary $\varphi \in \con1$,
  \begin{align}\label{gn}
    \int_\Omega \varphi^{p+1} 
    \le 
    C \| \nabla \varphi^\frac{p+m-1}{2} \|_{L^2(\Omega)}^{\frac{2(p+1)}{p+m-1}\theta}
    \| \varphi \|_{L^1(\Omega)}^{(p+1)(1-\theta)}
    + C \| \varphi \|_{L^1(\Omega)}^{p+1},
  \end{align}
where 
$\theta
\defs \frac{\frac{p+m-1}{2} - \frac{p+m-1}{2(p+1)}}
          {\frac{p+m-1}{2} + \frac{1}{n} - \frac{1}{2}} 
\in(0,1)$. 
\end{lemma}

\begin{proof} 
Applying the Gagliardo--Nirenberg inequality (e.g.\ in the variant of \cite[Lemma~2.3]{LiLankeit}, which allows for integrability exponents below $1$), with some $C>0$ we can estimate 
  \begin{align*}
    \int_\Omega \varphi^{p+1} 
    = \| \varphi^\frac{p+m-1}{2} \|_{L^\frac{2(p+1)}{p+m-1}(\Omega)}^{\frac{2(p+1)}{p+m-1}}
    &\le 
    C \| \nabla \varphi^\frac{p+m-1}{2} \|_{L^2(\Omega)}^{\frac{2(p+1)}{p+m-1}\theta}
    \| \varphi^\frac{p+m-1}{2} \|_{L^\frac{2}{p+m-1}(\Omega)}^{\frac{2(p+1)}{p+m-1}(1-\theta)}
    + 
    C \| \varphi^\frac{p+m-1}{2} \|_{L^\frac{2}{p+m-1}(\Omega)}^{\frac{2(p+1)}{p+m-1}}
    \\
    &=
    C \| \nabla \varphi^\frac{p+m-1}{2} \|_{L^2(\Omega)}^{\frac{2(p+1)}{p+m-1}\theta}
    \| \varphi \|_{L^1(\Omega)}^{(p+1)(1-\theta)}
    + 
    C \| \varphi \|_{L^1(\Omega)}^{p+1}
  \end{align*}
for all $\varphi \in \con1$, which concludes the proof.
\end{proof}

We next prove $L^p$-estimates for $u$ and $L^{p+1}$-estimates for $w$, provided $m \ge 2-\frac2n$.

\begin{lemma}\label{lem3.3}
Assume that $p>1$.
  \begin{itemize}
    \item[(i)] If $m=2-\frac{2}{n}$ and
    \begin{align}\label{def_mc}
      M=\intom u_0<M_c(p) \defs
      \left[ 
        \frac{1}{4 \cdot 2^p c_1} \cdot \frac{4(p-1)}{(p+m-1)^2} 
      \right]^\frac{1}{(1-\theta)(p+1)}
    \end{align}
    with $\theta$ from Lemma~\ref{GNineq}, then there exists $C>0$ such that
      \begin{align}\label{Lp}
        \| u(\cdot,t) \|_{L^p(\Omega)} \le C
        \quad\mbox{and}\quad
        \| w(\cdot,t) \|_{L^{p+1}(\Omega)} \le C
      \end{align}
    for all $t \in (0,\tmax)$.
    \item[(ii)] 
    In the case $m>2-\frac{2}{n}$, 
     there exists $C>0$ such that \eqref{Lp} holds.
  \end{itemize}
\end{lemma}

\begin{proof} 
We fix $k \defs 2 \cdot 2^p > 2^p$, so that  $k^{-p} + k^{-\frac{1}{p}}<1$.
Let us first consider the case (i). 
The condition $m=2-\frac{2}{n}$ ensures that $\theta$ defined in Lemma~\ref{GNineq} satisfies 
$\frac{2(p+1)}{p+m-1}\theta=2$. 
Thus, invoking Lemma~\ref{GNineq} and \eqref{mass}, we obtain $c_1>0$ such that with $M=\io u_0$ 
  \begin{align}\label{case(i)ineq}
    \int_\Omega u^{p+1} 
    \le 
    c_1 M^{(p+1)(1-\theta)} \| \nabla u^\frac{p+m-1}{2} \|_{L^2(\Omega)}^{2}
    + c_1 M^{p+1}
    \qquad \text{in $(0,\tmax)$}.
  \end{align}
If $M<M_c(p)$,
we can pick 
$\eta\in(0,1)$ so small that 
  \begin{align}\label{masscondi}
    M
    \le 
    \left[ 
      \frac{1}{2 ( k+\eta) c_1} \cdot \frac{4(p-1)}{(p+m-1)^2} 
    \right]^\frac{1}{(1-\theta)(p+1)}.
  \end{align}
By virtue of \eqref{case(i)ineq}, \eqref{masscondi} and Young's inequality,
in $(0,\tmax)$ we have 
  \begin{align*}
    2 k \int_\Omega u^{p+1} 
    &= 
    (2 k + \eta) \int_\Omega u^{p+1}
    - \eta \int_\Omega u^{p+1}
    \\
    &\le 
    (2 k + \eta) c_1 M^{(p+1)(1-\theta)} 
    \| \nabla u^\frac{p+m-1}{2} \|_{L^2(\Omega)}^{2}
    + (2 k + \eta) c_1 M^{p+1}
    - \eta \int_\Omega u^{p+1}
    \\
    &\le 
    \frac{4(p-1)}{(p+m-1)^2}
    \| \nabla u^\frac{p+m-1}{2} \|_{L^2(\Omega)}^{2}
    + c_2 
    - c_3 \int_\Omega u^p, 
  \end{align*}
where 
$c_2 \defs (2 k + \eta) c_1 M^{p+1} +  \frac{\eta}{p} |\Omega|$ 
and $c_3 \defs \frac{p+1}{p}\eta$. 
 We infer from \eqref{lem3.1;ineq} and the above inequality that 
  \begin{equation}\label{bothcasestogether}
    \ddt \set{\frac{1}{p} \int_\Omega u^p + \frac{1}{p+1} \int_\Omega w^{p+1}} 
    + c_3 \int_\Omega u^p
    + (1-( k^{-p} + k^{-\frac{1}{p}})) \int_\Omega w^{p+1}
    \le 
    c_2, 
  \end{equation}
in $(0,\tmax)$, 
which implies that 
  \[
    \ddt \set{\frac{1}{p} \int_\Omega u^p + \frac{1}{p+1} \int_\Omega w^{p+1}} 
    + c_4 \set{ \frac{1}{p} \int_\Omega u^p
    + \frac{1}{p+1} \int_\Omega w^{p+1}}
    \le 
    c_2
    \qquad \text{in $(0, \tmax)$},
  \]
where $c_4 \defs \min\{ c_3p,(1-( k^{-p} + k^{-\frac{1}{p}}))(p+1) \} >0$. 
From this differential inequality we arrive at \eqref{Lp}. 

Next, we deal with the case (ii). 
The condition $m>2-\frac{2}{n}$ 
implies $\frac{(p+1)}{p+m-1}\theta<1$, where $\theta$ is again as in Lemma~\ref{lem3.1;ineq1}.
Therefore, applying Young's inequality to 
the first term on the right-hand side of \eqref{gn}, 
we see that there exists $c_5>0$ such that 
  \begin{align*}
    \int_\Omega u^{p+1} 
    \le 
\frac{1}{2k+1} \cdot \frac{4(p-1)}{(p+m-1)^2}
    \| \nabla u^\frac{p+m-1}{2} \|_{L^2(\Omega)}^{2} 
    + c_5 \qquad\text{in } (0,\tmax)
  \end{align*}
and inserting this in \eqref{lem3.1;ineq}, we again achieve \eqref{bothcasestogether} (albeit with different constants) and conclude \eqref{Lp} exactly as in case~(i).
\end{proof}

As to the case~(i) in Lemma~\ref{lem3.3},  
the $L^p$-estimate independent of time for $u$ was shown 
by imposing a smallness condition on the size of initial data.
On the other hand, in the absence of such a smallness condition, 
we can derive a time-dependent $L^p$-estimate for $u$, 
which will allow us to establish global existence 
regardless of the size of initial data in the case $m=2-\frac{2}{n}$.

\begin{lemma}\label{lem3.4}
 If $m\ge 2-\frac{2}{n}$, then 
for all $p>1$ and $T \in (0,\tmax]\cap(0,\infty)$ 
there exists  $C>0$ such that 
  \[
    \| u(\cdot,t) \|_{L^p(\Omega)} \le C
    \quad\mbox{and}\quad
    \| w(\cdot,t) \|_{L^{p+1}(\Omega)} \le C   
  \]
for all $t\in(0,T)$.
\end{lemma}

\begin{proof} 
The case of $m>2-\f2n$ is already covered by Lemma~\ref{lem3.3}, we therefore 
assume $m=2-\f2n$ and note that this condition warrants that 
$\frac{2(p+1)}{p+m-1}\theta=2$ with $\theta$ as in Lemma~\ref{lem3.1;ineq1}.
Thanks to Lemma~\ref{GNineq} and H\"{o}lder's inequality, 
we can find $c_1>0$ such that 
  \begin{align*}
    \int_\Omega u^{p+1} 
    &\le 
    c_1 M^{(p+1)(1-\theta)} 
    \| \nabla u^\frac{p+m-1}{2} \|_{L^2(\Omega)}^2
    + c_1 M \kl{\int_\Omega u}^p
    \\ 
    &\le 
    c_1 M^{(p+1)(1-\theta)} 
    \| \nabla u^\frac{p+m-1}{2} \|_{L^2(\Omega)}^2
    + c_1 M |\Omega|^{p-1} \int_\Omega u^p
    \qquad \text{in $(0,\tmax)$},
  \end{align*}
which together with \eqref{lem3.1;ineq} entails that for any $k>0$, 
  \begin{align*}
    &\ddt \set{\frac{1}{p} \int_\Omega u^p + \frac{1}{p+1} \int_\Omega w^{p+1}} 
    + \frac{4(p-1)}{(p+m-1)^2} \int_\Omega |\nabla u^\frac{p+m-1}{2}|^2 
    + \int_\Omega w^{p+1}
    \notag \\
    &\quad\,\le 
    2  k c_1 M^{(p+1)(1-\theta)} 
    \int_\Omega |\nabla u^\frac{p+m-1}{2}|^2 
    + 2  k c_1 M |\Omega|^{p-1} \int_\Omega u^p
    +  (k^{-p} + k^{-\frac{1}{p}}) \int_\Omega w^{p+1}
    \qquad \text{in $(0,\tmax)$}.
  \end{align*}
Here, choosing 
$k=\frac{1}{2 c_1 M^{(p+1)(1-\theta)}} \cdot \frac{4(p-1)}{(p+m-1)^2}$, 
we can observe that 
    \begin{align*}
    &\ddt \set{\frac{1}{p} \int_\Omega u^p + \frac{1}{p+1} \int_\Omega w^{p+1}} 
    \le 
    c_2 \set{\frac{1}{p} \int_\Omega u^p
                + \frac{1}{p+1} \int_\Omega w^{p+1}}
    \qquad \text{in $(0,\tmax)$}
  \end{align*}
with 
$c_2 \defs \max\{ 2  k c_1 M|\Omega|^{p-1}p, 
                    (k^{-p} + k^{-\frac{1}{p}})(p+1) \}$. 
Therefore, 
for any $T\in(0,\tmax]\cap(0,\infty)$ 
we obtain $c_3(T)>0$ such that 
$\| u(\cdot,t) \|_{L^p(\Omega)} \le c_3(T)$ 
and 
$\| w(\cdot,t) \|_{L^{p+1}(\Omega)} \le c_3(T)$ 
for all $t\in(0,T)$.
\end{proof}

Now we prove global existence and boundedness of solutions to \eqref{P} 
by applying Lemmata~\ref{lem3.3} and \ref{lem3.4}.

\begin{proof}[Proof of Theorem \ref{GB}]
Let $p>n+2$. 
We first show that $\tmax=\infty$. 
Let us fix $T\in(0,\tmax]\cap(0,\infty)$.  
Thanks to Lemma~\ref{lem3.4}, there exists $c_1(p,T)>0$ satisfying
$\|u(\cdot,t)\|_{L^p(\Omega)}\le  c_1(p,T)$ and 
$\|w(\cdot,t)\|_{L^{p+1}(\Omega)}\le  c_1(p,T)$ 
for all $t\in(0,T)$. 
Therefore, according to elliptic regularity theory (cf.\ \cite[Theorem~I.19.1]{Friedman})
there exists  $c_2(p,T)>0$ such that
$\| v(\cdot,t) \|_{W^{2,p+1}(\Omega)}\le  c_2(p,T)$
for all $t\in(0,T)$, 
and then the Sobolev embedding theorem tells us that 
$\| \nabla v(\cdot,t) \|_{L^\infty(\Omega)} \le  c_3(p,T)$
for all $t\in(0,T)$ with some $c_3(p,T)>0$. 
Hence, applying the Moser-type iteration 
of \cite[Lemma A.1]{Tao-W_2012_JDE}, 
we can obtain $c_4(p,T)>0$ such that 
  \[
    \| u(\cdot,t) \|_{L^\infty(\Omega)}\le c_4(p,T)
  \]
for all $t\in(0,T)$. This in conjunction with \eqref{criterion} 
yields $\tmax=\infty$. 

In the case $m>2-\frac{2}{n}$ or if $M < M_c \defs M_c(n+3)$ (with $M_c(n+3)$ as in \eqref{def_mc}, which crucially does not depend on $M$ or the solution $(u, v, w)$) in the case of $m=2-\f2n$, by means of Lemma~\ref{lem3.3}, 
elliptic regularity theory and the Sobolev embedding theorem 
we can similarly verify that 
$\| \nabla v(\cdot,t) \|_{L^\infty(\Omega)}\le c_5$
for all $t\in(0,\tmax)$ with some $c_5>0$, 
where $c_5$ is independent of time. 
Thus, by a Moser-type iteration we see that 
$\| u(\cdot,t) \|_{L^\infty(\Omega)}\le c_6$ 
for all $t\in(0,\tmax)$ with some $c_6>0$, and so
\eqref{bdd} holds. 
\end{proof}

\section{Unboundedness}\label{sec4}
In the following, we let 
$\Omega \defs  B_1(0)\subset\mathbb{R}^n$ $(n\ge3)$, $m \in [1,2-\frac2n]$ and $M > 0$.
For simplicity, we also fix nonnegative radially symmetric initial data $u_0 \in C^0(\Ombar)$ and $w_0 \in C^1(\Ombar)$
as well as the solution $(u, v, w)$ of \eqref{P} given by Proposition~\ref{local} and denote its maximal existence time by $\tmax$.
However, we emphasize that all constants below only depend on $\Omega$, $m$ and $M$, not explicitly on the initial data or the solution.
Moreover, the uniqueness statement in Proposition~\ref{local} implies that $(u, v, w)$ is radially symmetric
and henceforth we write $u(|x|, t)$ instead of $u(x, t)$ etc.

In order to prove Theorem \ref{BU}, 
referring to \cite{Tao-W_2017_JEMS},
we define the function $U$ as 
  \begin{align}\label{defU}
    U(\xi,t) \defs  \int^{\xi^\frac{1}{n}}_0 r^{n-1}u(r,t)\dr
    \quad\mbox{for}\ \xi\in[0,1]\ \mbox{and}\ t\in[0,\tmax),
  \end{align}
which belongs to 
$C^{1,0}([0,1]\times[0,\tmax))
\cap C^{2,1}([0,1]\times(0,\tmax))$,
and introduce the parabolic operator $\mc{P}$ as 
  \begin{align}\label{defP}
    \mc{P}\wt{U}(\xi,t)
    &\defs 
    \wt{U}_t(\xi,t) 
    - n^2 \xi^{2-\frac{2}{n}} (n\wt{U}_\xi(\xi,t)+1)^{m-1}\wt{U}_{\xi\xi}(\xi,t)
    - n \set{\int^t_0 \ure^{-(t-s)} \kl{\wt{U}(\xi,s)-\frac{M}{\omega_n}\xi} \ds}\wt{U}_\xi(\xi,t)
    \notag\\
    &\quad\,
    - n(W_0(\xi)-K_0\xi)\ure^{-t}\wt{U}_\xi(\xi,t)
  \end{align}
for $\xi\in(0,1)$, 
 $t\in(0,T)$ and 
$\wt{U} \in C^1((0,1)\times(0,T)) \cap C^0((0,T);W^{2,\infty}((0,1)))$, $T > 0$,
with 
  \begin{align}\label{W0K0}
    W_0(\xi) \defs  \int^{\xi^\frac{1}{n}}_0 r^{n-1}w_0(r)\dr
    \quad\mbox{for}\ \xi\in[0,1] 
    \quad\mbox{and}\quad
    K_0\defs W_0(1),
  \end{align}
where $\omega_n=|\partial B_1(0)| = n|B_1(0)|$.
Now we first collect properties of $U$.

\begin{lemma}
The function $U$ satisfies that
  \begin{align}\label{pro1}
    U(0,t) = 0 
    \quad\mbox{and}\quad
    U(1,t) = \frac{M}{\omega_n}
  \end{align}
for all $t\in[0,\tmax)$ as well as
  \begin{align}\label{pro2}
     U_\xi(\xi,t)=\frac{1}{n}u(\xi^\frac{1}{n},t) \ge 0
  \end{align}
for all $\xi\in(0,1)$ and $t\in(0,\tmax)$. Moreover, 
  \begin{align}\label{PU}
    \mc{P}U(\xi,t)=0
  \end{align}
for all $\xi\in(0,1)$ and $t\in(0,\tmax)$.
\end{lemma}

\begin{proof} 
We immediately see that \eqref{pro1} holds 
from the definition of $U$ and \eqref{mass},
and that \eqref{pro2} is obtained by a direct computation and nonnegativity of $u$.
Also, transforming the system \eqref{P} 
exactly as in \cite[Lemma 4.1]{Tao-W_2017_JEMS} (there only for 2D and linear diffusion), 
we arrive at \eqref{PU}. 
\end{proof}

As a preparation to the proof of Theorem~\ref{BU}, 
let us prove a comparison principle. 
Before stating the result, we introduce 
the functions $A,B,D$ such that for arbitrary $T>0$,
\begin{align}\label{reg_abd}
  A \in C^0((0,1)\times[0,T)\times[0,\infty)),\ 
  B \in C^0((0,1)\times[0,T)),\ 
  D \in C^0([0,1]\times[0,T]\times[0,T])
\end{align}
and the operator $\mc{Q}$ 
such that
\begin{align}\label{def_Q}
  \mc{Q}\wt{U}(\xi,t)
  \defs 
  \wt{U}_t(\xi,t) 
  - A(\xi,t,\wt{U}_\xi)\wt{U}_{\xi\xi}(\xi,t)
  - \set{B(\xi,t)+\int^t_0 D(\xi,t,s)\wt{U}(\xi,s) \ds}\wt{U}_\xi(\xi,t)
\end{align}
for $\xi\in(0,1)$, $t\in[0,T)$ and sufficiently regular $\wt{U}:(0,1)\times(0,T)\to\mathbb{R}$. 
We note that the operator $\mc{Q}$ slightly differs from the definition in \cite[(4.9)]{Tao-W_2017_JEMS},
where $A$ depends on $\xi$ and $t$ only.

\begin{lemma}\label{lem4.2}
Let $t_1\ge0$ and $T>t_1$. Suppose that $A$, $B$ and $D$ 
satisfy \eqref{reg_abd},
  \[
    A\ge0
    \quad\mbox{in}\ (0,1)\times(t_1,T)\times[0,\infty) 
    \quad\mbox{and}\quad
    D\ge0
    \quad\mbox{in}\ [0,1]\times[0,T]\times[0,T].
  \]
 Moreover, assume that 
  \[
    \ul{U}, \ol{U} \in 
    C^0([0,1]\times[0,T]) 
    \cap C^1((0,1)\times(t_1,T)) 
    \cap C^0((t_1,T);W^{2,\infty}((0,1)))
  \]
are nonnegative and such that
  \[
    0 \le \ul{U}_\xi(\xi,t) \le L
    \quad\mbox{for all}\ \xi\in(0,1)\ \mbox{and}\ t\in(t_1,T)
  \]
with some $L>0$ and such that
  \[
    \mc{Q}\ul{U}(\xi,t) \le \mc{Q}\ol{U}(\xi,t) 
    \quad\mbox{for a.e.}\ \xi\in (0,1)\ \mbox{and all}\ t\in(t_1,T),
  \]
where $\mc Q$ is as in \eqref{def_Q}.
If 
  \begin{align}\label{ulol1}
    \ul{U}(\xi,t) \le \ol{U}(\xi,t)
    \quad\mbox{for all}\ \xi\in[0,1]\ \mbox{and}\ t\in[0,t_1] 
  \end{align}
and 
  \begin{align}\label{ulol2}
    \ul{U}(0,t) \le \ol{U}(0,t)
    \quad\mbox{for all}\ t\in[t_1,T], 
    \quad
    \ul{U}(1,t) \le \ol{U}(1,t) 
    \quad\mbox{for all}\ t\in[t_1,T],
  \end{align}
then
  \[
    \ul{U}(\xi,t) \le \ol{U}(\xi,t)
    \quad\mbox{for all}\ \xi\in[0,1]\ \mbox{and}\ t\in[0,T]. 
  \]
\end{lemma}

\begin{proof} 
 This can be shown as in the proof of \cite[Lemma 4.2]{Tao-W_2017_JEMS}, with the only difference being that in our setting $A$ also depends on $\wt{U}_\xi$.
The main idea is to show that for arbitrary $\eps > 0$ and certain $\beta > 0$, the function
\[
  d(\xi,t) \defs  \ul{U}(\xi,t) - \ol{U}(\xi,t) - \eps \ure^{\beta t} \quad 
  \mbox{for}\ \xi\in[0,1] \ \mbox{and}\ t\in[0,T)
\]
is negative. If this were not the case, then \eqref{ulol1} and \eqref{ulol2} would ensure that there exists $(\xi_\star, t_\star) \in (0, 1) \times (t_1, T)$ such that $d(\xi_\star, t_\star) = 0$.
As we could further assume $t_\star$ to be minimal, $d(\cdot, t_\star)$ would attain a maximum at $\xi_\star$, hence $d_\xi(\xi_\star, t_\star) = 0$.
Since then $A(\xi_\star, t_\star, \ul U_\xi(\xi_\star, t_\star)) = A(\xi_\star, t_\star, \ol U_\xi(\xi_\star, t_\star))$ and due to the required regularity of $A$, $\ul U$ and $\ol U$,
the fact that $A$ depends on $\wt U_\xi$ turns out to cause no additional challenges and we can arrive at a contradiction just as in the proof of \cite[Lemma 4.2]{Tao-W_2017_JEMS}.
\end{proof}

To prove blow-up, we shall construct a subsolution with unbounded space derivative to \eqref{PU}. 
Referring to \cite[(6.1)]{Tao-W_2017_JEMS}, we put the function 
$\ul{U}$ as 
  \begin{align}\label{ulU}
    \ul{U}(\xi,t)
    \defs  
    \begin{cases}
      \dfrac{a(t)\xi}{b(t)+\xi}
      &\mbox{if}\ \xi\in[0,\xi_0]\ \mbox{and}\ t\in[0,\infty),
    \\[5mm]
      \dfrac{a(t)b(t)\xi+a(t)\xi_0^2}{(b(t)+\xi_0)^2}
      &\mbox{if}\ \xi\in(\xi_0,1]\ \mbox{and}\ t\in[0,\infty),
    \end{cases}
  \end{align}
where $\xi_0\in(0,1)$, and the functions $a$ and $b$ are defined as 
  \begin{align}\label{defab}
    a(t) \defs  \frac{M}{\omega_n} \cdot \frac{(b(t)+\xi_0)^2}{b(t)+\xi_0^2}
    \quad\mbox{and}\quad
    b(t) \defs  b_0 \ure^{-\alpha t} 
    \qquad\mbox{for}\ t\in[0,\infty)
  \end{align}
with some $b_0>0$ and $\alpha>0$. 
Then the function $\ul{U}$ satisfies (cf.\ \cite[Lemma 6.1]{Tao-W_2017_JEMS})
  \[
    \ul{U} \in C^1([0,1]\times[0,\infty)) 
                  \cap C^0([0,\infty);W^{2,\infty}((0,1)))
                  \cap C^0([0,\infty);C^2_\ur{loc}([0,1]\setminus\{\xi_0\})).
  \]
Moreover, by computing $\mc{P}\ul{U}$, we have the following lemma.

\begin{lemma}
 Let $\alpha, b_0 > 0$, $\xi_0 \in (0, 1)$ and $a, b$ be defined as in \eqref{defab}. Then the function $\ul{U}$ defined in \eqref{ulU}
satisfies that 
  \begin{align}\label{PulU1}
    \frac{(b(t)+\xi)^2}{a(t)b(t)\xi} \mc{P}\ul{U}(\xi,t) 
    &= 
    \frac{a'(t)(b(t)+\xi)}{a(t)b(t)} - \frac{b'(t)}{b(t)}
    + 2n^2 \kl{\frac{na(t)b(t)}{(b(t)+\xi)^2}+1}^{m-1}\frac{\xi^{1-\frac{2}{n}}}{b(t)+\xi}
    \notag\\ 
    &\quad\,
    - n \int^t_0 \ure^{-(t-s)} \kl{\frac{a(s)}{b(s)+\xi} - \frac{M}{\omega_n}} \ds
    - n \kl{\frac{W_0(\xi)}{\xi} - K_0} \ure^{-t}
  \end{align}
for all $\xi\in(0,\xi_0)$ and $t\in(0,\infty)$ and 
  \begin{align}\label{PulU2}
    \frac{(b(t)+\xi_0)^2}{a(t)b(t)} \mc{P}\ul{U}(\xi,t)
    &=
    \frac{a'(t)\xi}{a(t)} + \frac{b'(t)\xi}{b(t)} + \frac{a'(t)\xi_0^2}{a(t)b(t)} 
    - 2 \frac{b'(t)\xi + \frac{b'(t)}{b(t)}\xi_0^2}{b(t)+\xi_0}
    \notag\\ 
    &\quad\,
    - n \int^t_0 \ure^{-(t-s)} \kl{\frac{a(s)b(s)\xi + a(s)\xi_0^2}{(b(s)+\xi_0)^2} - \frac{M}{\omega_n}\xi} \ds
    \notag\\ 
    &\quad\,
    - n(W_0(\xi) - K_0\xi) \ure^{-t}
  \end{align}
for all $\xi\in(\xi_0,1)$ and $t\in(0,\infty)$.
\end{lemma}

\begin{proof} 
This results from straightforward computations; 
for details see \cite[Lemma 6.1]{Tao-W_2017_JEMS}, 
where only the third term in the right-hand side of \eqref{PulU1} differs slightly from the one of \cite[(6.2)]{Tao-W_2017_JEMS}. 
\end{proof}

Our goal is to make sure that $\mc{P}\ul{U}\le0$ in $(0,1)\times(0,\infty)$. 
Since the function $\ul{U}$ fulfills \eqref{PulU2} in $(\xi_0,1)\times(0,\infty)$ 
which is similar to \cite[(6.3)]{Tao-W_2017_JEMS}, 
we can immediately obtain the following lemma.

\begin{lemma}\label{lem4.4}
 Assume that $\xi_0\in(0,1)$ and $\eta_0>0$ satisfy
  \[
    \frac{W_0(\xi) - K_0\xi}{1-\xi} \ge \eta_0
  \]
for all $\xi\in(\xi_0,1)$. 
Then for all $\alpha_\star>0$ there exists $\alpha\in(0,\alpha_\star)$ 
such that for any choice of $b_0\in(0,\xi_0^2)$, 
the function $\ul{U}$ in \eqref{ulU} satisfies 
  \[
    \mc{P}\ul{U}(\xi,t) \le 0
  \]
for all $\xi\in(\xi_0,1)$ and $t\in(0,\infty)$.
\end{lemma}

\begin{proof} 
The main difference of the operator $\mc P$ introduced in \eqref{defP} compared to the one considered in \cite{Tao-W_2017_JEMS} is the factor in front of $\widetilde U_{\xi\xi}$,
which of course is inconsequential in regions where $\ul U_{\xi\xi}$ vanishes. Accordingly, this lemma can be shown as in \cite[Lemma~6.2 and Lemma~6.3]{Tao-W_2017_JEMS},
with a slightly different choice of $\eta_0$ depending on $n$.
\end{proof}

Before estimating the right-hand side of \eqref{PulU1}, we recall an estimate for 
its first term.

\begin{lemma}\label{lem4.5}
 Let $\alpha,b_0>0$, $\xi_0\in(0,1)$ and $a,b$ be defined as in \eqref{defab}. 
The function $a$ satisfies 
  \[
    \frac{a'(t)(b(t)+\xi)}{a(t)b(t)} \le \frac{\alpha}{\xi_0}
  \]
for all $\xi\in(0,\xi_0)$ and $t\in(0,\infty)$.
\end{lemma}

\begin{proof} 
This is derived from the definitions of $a$ and $b$ in \eqref{defab}, 
for the computation see \cite[Lemma 6.4]{Tao-W_2017_JEMS}.
\end{proof}

Now we prove the estimate $\mc{P}\ul{U}\le0$ in $(0,\xi_0)\times(0,\infty)$ with some $\xi_0\in(0,1)$.
We first consider this estimate for suitably large times. 
When $m=2-\frac{2}{n}$, a largeness condition for $M$ is needed, 
whereas when $m\in\big[1,2-\frac{2}{n}\big)$, 
the estimate $\mc{P}\ul{U}\le0$ is satisfied for any $M>0$.

\begin{lemma}\label{lem4.6}
Assume that 
  \begin{align}\label{lem4.6;condi}
    W_0(\xi) - K_0\xi \ge 0
  \end{align}
for all $\xi\in(0,1)$. 
\begin{itemize}
 \item[(i)] If $m=2-\f2n$ and
       \begin{align}\label{BUmass}
        M > 2^\frac{n}{2}n^{n-1}\omega_n,
      \end{align}
or
 \item[(ii)] if $m\in\big[1,2-\f2n\big)$ and $M>0$, 
\end{itemize}
then 
there exist $\xi_0\in(0,1)$ and $\alpha_\star>0$ 
    with the property that 
    for all $\alpha\in(0,\alpha_\star)$ one can find $b_0\in(0,\xi_0^2)$
    and $t_0\in(0,\infty)$ such that the function $\ul{U}$ in \eqref{ulU} 
    satisfies 
      \begin{align*}
        \mc{P}\ul{U}(\xi,t) \le 0
      \end{align*}
    for all $\xi\in(0,\xi_0)$ and $t\in[t_0,\infty)$. 
\end{lemma}

\begin{proof} 
In part, this proof is similar to \cite[Lemma~6.5]{Tao-W_2017_JEMS}. However, we choose to still give a full proof for two reasons:
First, this is the point where the conditions on $m$ and $M$ play a crucial role.
Second, unlike in same of the proofs above, here we need to introduce new ideas for dealing with the nonlinear diffusion present in \eqref{P} (and hence in \eqref{lem4.2}) for $m > 1$.

We assume that $m$ and $M$ are as required by (i) or (ii)
and take $\eps\in(0,1)$ and 
$\xi_0\in(0,1)$ satisfying
  \begin{align}\label{xi01}
    \xi_0 \le \frac{\eps}{2}
  \end{align}
and, in the case of (ii), also 
  \begin{align}\label{xi02}
    \xi_0^{2-\frac{2}{n}-m} 
    \le 
    \frac{(1-\eps)^3 n M}{(1+\eps)\omega_n} \cdot \frac{1}{2n^2}
    \kl{(\eps+1)^2\frac{nM}{\omega_n}+\frac{\eps}{2}}^{-(m-1)},
  \end{align}
so that, in both cases, 
\[
 c_1\defs  \frac{(1-\eps)^3 n M}{(1+\eps)\omega_n}
    - 2n^2\kl{(\eps+1)^2\frac{nM}{\omega_n}+\frac{\eps}{2}}^{m-1} 
    \xi_0^{2-\frac{2}{n}-m} >0, 
\]
either by \eqref{BUmass} and sufficiently small choice of $ε$ or by \eqref{xi02}.
We let 
  \begin{align}\label{astar1}
    α_\star=\min\set{\frac{\log{\frac{1}{1-\eps}}}{\log{\frac{1}{\eps}}},\f{c_1}4}.
  \end{align}
Moreover, given $\alpha\in(0,\alpha_\star)$, 
we pick $b_0>0$ and $t_0>0$ fulfilling that 
  \begin{align}\label{b0}
    b_0 \le \eps\xi_0^2 \le \xi_0^2 \le \xi_0
  \end{align}
and
  \begin{align}\label{t0}
    t_0 \ge \frac{1}{\alpha} \log{\frac{1}{1-\eps}}.
  \end{align}
The condition \eqref{lem4.6;condi} and the identity \eqref{PulU1} assert that 
  \begin{align}\label{PulU1-2}
    \frac{(b(t)+\xi)^2}{a(t)b(t)\xi} \mc{P}\ul{U}(\xi,t) 
    &\le
    \frac{a'(t)(b(t)+\xi)}{a(t)b(t)} - \frac{b'(t)}{b(t)}
    + 2n^2 \kl{\frac{na(t)b(t)}{(b(t)+\xi)^2}+1}^{m-1}\frac{\xi^{1-\frac{2}{n}}}{b(t)+\xi}
    \notag\\ 
    &\quad\,
    - n \int^t_0 \ure^{-(t-s)} \kl{\frac{a(s)}{b(s)+\xi} - \frac{M}{\omega_n}} \ds
  \end{align}
for all $\xi\in(0,\xi_0)$ and $t\in(0,\infty)$. 
It follows from Lemma~\ref{lem4.5} and the definition of $b$ in \eqref{defab} that 
  \begin{align}\label{J1}
    J_1(t)\defs \frac{a'(t)(b(t)+\xi)}{a(t)b(t)} - \frac{b'(t)}{b(t)}\le \f{\alpha}{\xi_0} + \alpha \le \frac{2}{\xi_0}\alpha
  \end{align}
and from the definition of $a$ in \eqref{defab} and \eqref{xi01} (and \eqref{b0}) that
  \begin{align}\label{a1}
      \eps \frac{a(t)}{b(t)+\xi}
  = \frac{M}{\omega_n} \cdot \frac{(b(t)+\xi_0)^2 \eps}{(b(t)+\xi)(b(t)+\xi_0^2)}
  \ge \frac{M}{\omega_n} \cdot \frac{(b(t)+\xi_0) \eps}{b(t)+\xi_0^2}
  \ge \frac{M}{\omega_n} \cdot \frac{\xi_0 \cdot 2 \xi_0}{\xi_0^2+\xi_0^2}
  =   \frac{M}{\omega_n} 
  \end{align}
as well as from \eqref{defab} and \eqref{b0} that 
  \begin{align*}
    a(t)=\f{M}{\omega_n} \cdot \f{(b+\xi_0)^2}{b(t)+\xi_0^2}
    \ge \f{M}{\omega_n}  \cdot \f{\xi_0^2}{b_0+\xi_0^2} 
    \ge \frac{M}{(1+\eps)\omega_n}
  \end{align*}
for all $\xi\in(0,\xi_0)$ and $t\in(0,\infty)$ (see also \cite[p.\,3671]{Tao-W_2017_JEMS}). 
Furthermore, as in \cite[pp.\,3671--3672]{Tao-W_2017_JEMS} 
the conditions \eqref{astar1} and \eqref{t0} ensure that 
  \begin{align}\label{int}
    \int^t_0 \ure^{-(t-s)} \frac{1}{b(s)+\xi} \ds 
    \ge \frac{(1-\eps)^2}{b(t)+\xi}
  \end{align}
for all $\xi\in(0,\xi_0)$ and $t\in[t_0,\infty)$. 
Therefore, in light of \eqref{a1}--\eqref{int} we obtain 
  \begin{align}\label{J2}
    - n \int^t_0 \ure^{-(t-s)} \kl{\frac{a(s)}{b(s)+\xi} - \frac{M}{\omega_n}} \ds
    &\le
    - (1-\eps)n \int^t_0 \ure^{-(t-s)} \frac{a(s)}{b(s)+\xi} \ds 
    \notag\\ 
    &\le 
    - \frac{(1-\eps)nM}{(1+\eps)\omega_n} \int^t_0 \ure^{-(t-s)} \frac{1}{b(s)+\xi} \ds 
    \notag\\ 
    &\le 
    - \frac{(1-\eps)^3nM}{(1+\eps)\omega_n} \cdot \frac{1}{b(t)+\xi} 
  \end{align}
for all $\xi\in(0,\xi_0)$ and $t\in[t_0,\infty)$. 
Also, since \eqref{defab}, \eqref{b0} and the relation $\xi_0<1$ 
yield 
  \[
    a(t) 
    \le \frac{M}{\omega_n} \cdot \frac{(\eps\xi_0^2+\xi_0)^2}{\xi_0^2}
    \le (\eps+1)^2 \cdot \frac{M}{\omega_n} 
  \]
for all $t\in[0,\infty)$, 
we infer from this inequality, the fact that $\frac{b(t)}{b(t)+\xi}\le1$ 
and \eqref{xi01} that
  \begin{align*}
    2n^2 \kl{\frac{na(t)b(t)}{(b(t)+\xi)^2}+1}^{m-1}\frac{\xi^{1-\frac{2}{n}}}{b(t)+\xi}
    &\le
    2n^2 \kl{(\eps+1)^2 \cdot \frac{nM}{\omega_n} \cdot \frac{1}{b(t)+\xi}+1}^{m-1}\frac{\xi^{1-\frac{2}{n}}}{b(t)+\xi}
    \\
    &\le
    2n^2 \kl{(\eps+1)^2 \cdot \frac{nM}{\omega_n} \cdot \frac{1}{\xi}+1}^{m-1}\frac{\xi^{1-\frac{2}{n}}}{b(t)+\xi}
    \\
    &=
    2n^2 \kl{(\eps+1)^2 \cdot \frac{nM}{\omega_n}+\xi}^{m-1}\frac{\xi^{2-\frac{2}{n}-m}}{b(t)+\xi}
    \\
    &\le
    2n^2 \kl{(\eps+1)^2 \cdot \frac{nM}{\omega_n}+\frac{\eps}{2}}^{m-1}\frac{\xi_0^{2-\frac{2}{n}-m}}{b(t)+\xi}
  \end{align*}
for all $\xi\in(0,\xi_0)$ and $t\in[0,\infty)$,
which together with \eqref{J2} implies that 
  \begin{align}\label{J2-2}
     J_2(t)
    & \defs 
    2n^2 \kl{\frac{na(t)b(t)}{(b(t)+\xi)^2}+1}^{m-1}\frac{\xi^{1-\frac{2}{n}}}{b(t)+\xi}
    - n \int^t_0 \ure^{-(t-s)} \kl{\frac{a(s)}{b(s)+\xi} - \frac{M}{\omega_n}} \ds
    \notag\\ 
    &\le
    \left[2n^2 \kl{(\eps+1)^2 \cdot \frac{nM}{\omega_n} +\frac{\eps}{2}}^{m-1}\xi_0^{2-\frac{2}{n}-m}
     - \frac{(1-\eps)^3nM}{(1+\eps)\omega_n}\right] \cdot \frac{1}{b(t)+\xi} =-\f{c_1}{b(t)+ξ}
  \end{align}
for all $\xi\in(0,\xi_0)$ and $t\in[t_0,\infty)$. 
Noting from \eqref{b0} that 
$b(t)+\xi \le \eps\xi_0^2 + \xi_0 \le 2\xi_0$,  
we have from \eqref{PulU1-2}, \eqref{J1} and \eqref{J2-2} that 
  \[
    \frac{(b(t)+\xi)^2}{a(t)b(t)\xi} \mc{P}\ul{U}(\xi,t) 
    \le \frac{2}{\xi_0}\alpha-\frac{c_1}{b(t)+\xi}
    \le \frac{2}{\xi_0}\alpha-\frac{c_1}{2\xi_0}
    =\frac{2}{\xi_0}\kl{\alpha-\frac{c_1}{4}} \le0
  \]
for all $\xi\in(0,\xi_0)$ and $t\in[t_0,\tmax)$. 
\end{proof}
Next, we derive the estimate $\mc{P}\ul{U}\le0$ near the origin.

\begin{lemma}\label{lem4.7}
Let $\alpha,b_0>0$ and $\xi_0\in(0,1)$. 
Then, for all $t_0\in(0,\tmax)$ there exists 
$\Gamma_0>0$
such that whenever $W_0$ and $K_0$ satisfy 
  \[
    \frac{W_0(\xi)}{\xi} - K_0 \ge  \Gamma_0
  \]
for all $\xi\in(0,\xi_0)$, 
the function $\ul{U}$ in \eqref{ulU} satisfies 
  \[
    \mc{P}\ul{U}(\xi,t) \le 0
  \]
for all $\xi\in(0,\xi_0)$ and $t\in(0,t_0)$.
\end{lemma}

\begin{proof} 
In estimating the terms on the right-hand side of \eqref{PulU1}, 
we follow \cite[Lemma 6.6]{Tao-W_2017_JEMS}. 
In the following, we mainly point out the differences. 
Since we have from \eqref{defab}
that 

  \[
    \frac{b(t)}{(b(t)+\xi)^2}\le \frac{1}{b(t)}=\frac{\ure^{\alpha t}}{b_0}
  \quad\mbox{and}\quad 
    a(t) 
    \le \frac{M}{\omega_n} \cdot \frac{(b_0+1)^2}{b(t)}
    = \frac{M}{\omega_n} \cdot \frac{(b_0+1)^2}{b_0}\ure^{\alpha t},
  \] 
it follows that 
  \begin{align*} 
    \frac{na(t)b(t)}{(b(t)+\xi)^2} 
    \le \frac{nM}{\omega_n} \cdot \frac{(b_0+1)^2}{b_0^2}\ure^{2\alpha t_0} \sfed c_1
  \end{align*}
for all $\xi\in(0,\xi_0)$ and $t\in(0,t_0)$. 
Thus we can estimate 
  \begin{align*}
    2n^2 \kl{\frac{na(t)b(t)}{(b(t)+\xi)^2}+1}^{m-1}\frac{\xi^{1-\frac{2}{n}}}{b(t)+\xi}
    &\le 2n^2 (c_1+1)^{m-1}\cdot\frac{\xi^{1-\frac{2}{n}}}{b(t)+\xi}
    \\
    &\le 2n^2 (c_1+1)^{m-1}\cdot\frac{\xi_0^{1-\frac{2}{n}}}{b_0}\ure^{\alpha t_0}\sfed c_2
  \end{align*}
for all $\xi\in(0,\xi_0)$ and $t\in(0,t_0)$. 
Therefore, by Lemma~\ref{lem4.5} and \eqref{defab}
\[
 \f{(b+\xi)^2}{ab\xi}\mc{P}\ul{U} \le \f{\alpha}{\xi_0} + \alpha + c_2 - n\int_0^t \ure^{-(t-s)}\left(\f{a(s)}{b(s)+\xi}-\f{M}{\omega_n}\right) \ds - n \left(\f{W_0}{\xi}-K_0\right) \ure^{-t}.
\]
As in \cite[Lemma 6.6]{Tao-W_2017_JEMS}, $- n\int_0^t \ure^{-(t-s)}\big(\f{a}{b+\xi}-\f{M}{\omega_n}\big) \ds\le 0$, therefore $\mc{P}\ul{U}\le 0$ for all $\xi\in(0,\xi_0)$ and $t\in(0,t_0)$ if we set 
  \[
     \Gamma_0\defs \frac{1}{n}\kl{\kl{\frac{1}{\xi_0}+1}\alpha + c_2
    }\ure^{t_0}.
    \qedhere
  \]
\end{proof}

Finally, we construct an initial data to show Theorem~\ref{BU}.

\begin{lemma}\label{lem4.8}
Let $\alpha, b_0, \eta, \Gamma_0 > 0$, $\xi_0 \in (0, 1)$ and $a, b$ be defined as in \eqref{defab}. 
Set $R\defs \xi_0^\frac{1}{n}$, $\eta_0\defs \frac{\eta}{n}$ as well as
  \begin{align}\label{defgs}
    \Gamma_u
    &\defs \f{na(0)}{b(0)}=\frac{nM}{\omega_n}\cdot\frac{(b_0+\xi_0)^2}{b_0+\xi_0^2}\cdot \f1{b_0}, \quad
    \gamma
    \defs \frac{nM}{\omega_n}\cdot\frac{b_0}{b_0+\xi_0^2}
\quad\text{and}\quad
    \Gamma_w
    \defs n\Gamma_0.
  \end{align}
If $u_0$ and $w_0$ satisfy \eqref{BUcondi2}--\eqref{BUcondi5}, then 
  \begin{align}\label{condiw0}
    \frac{W_0(\xi) - K_0\xi}{1-\xi} \ge \eta_0
    \quad\mbox{for all}\ \xi\in(\xi_0,1)
    \quad\mbox{and}\quad
    \frac{W_0(\xi)}{\xi} - K_0 \ge \Gamma_0
    \quad\mbox{for all}\ \xi\in(0,\xi_0), 
  \end{align}
and moreover, the function $\ul{U}$ defined in \eqref{ulU}
satisfies that 
  \begin{align}\label{ulU0U0}
    \ul{U}(\xi,0) \le U(\xi,0)
    \quad\mbox{for all}\ \xi\in(0,1).
  \end{align}
\end{lemma}
\begin{proof} 
In the same way as in \cite[p.\,3674]{Tao-W_2017_JEMS}, 
from \eqref{W0K0} and \eqref{BUcondi5} we can make sure that 
$W_0(\xi) - K_0\xi 
=(1-\xi)\big(K_0 - \f1n\fint_{B_1\setminus B_{\xi^{\f1n}}} w_0\big)
\ge (1-\xi) \frac{\eta}{n}=(1-\xi)\eta_0$ for all $\xi\in(\xi_0,1)$,
that is, the first inequality in \eqref{condiw0} holds. 
Moreover, condition \eqref{BUcondi4} implies that 
$\frac{W_0(\xi)}{\xi} - K_0 
\ge \frac{1}{n} \Gamma_w = \Gamma_0$
for all $\xi\in(0,\xi_0)$, see also \cite[(6.43)]{Tao-W_2017_JEMS}. 
Finally, as in \cite[p.\,3675]{Tao-W_2017_JEMS}, combining the definitions \eqref{ulU} and \eqref{defab} with the condition \eqref{BUcondi2}, 
we find that 
$\ul{U}(\xi,0) 
\le \frac{1}{n} \Gamma_u \xi
\le U(\xi,0)$ for all $\xi\in(0,\xi_0)$
and from \eqref{ulU}, \eqref{defab} and \eqref{BUcondi3} and the definition of $\gamma$ that 
$\ul{U}(\xi,0) 
    \le \frac{M}{\omega_n} 
         - \frac{M}{\omega_n} \cdot \frac{b_0}{b_0+\xi_0^2}(1-\xi)
    \le U(\xi,0)$
for all $\xi\in(\xi_0,1)$.%
\end{proof}%

We are now in position to prove Theorem~\ref{BU}.

\begin{proof}[Proof of Theorem~\ref{BU}] 
This theorem can be shown similarly 
to \cite[Theorem 1.3]{Tao-W_2017_JEMS}. 
Thus we only give a sketch of the proof. 
First, we may assume $\tmax=\infty$ as the case $\tmax<\infty$ is already covered by Proposition~\ref{local}.
We then take $\xi_0\in(0,1)$, $\alpha_\star>0$, 
$b_0\in(0,\xi_0^2)$ and $t_0\in(0,\infty)$ given by Lemma~\ref{lem4.6}, 
and next set $R\defs \xi_0^\frac{1}{n}$.  
Also, we put $\eta_0\defs \frac{\eta}{n}$ 
and pick $\alpha\in(0,\alpha_\star)$ provided in Lemma~\ref{lem4.4}. 
Moreover, we define 
$\Gamma_u, \gamma$ and $\Gamma_w$ as in \eqref{defgs}. 
 Our goal is to prove that 
  \begin{align}\label{goal}
    u(0,t)=nU_\xi(0,t) \ge n\ul{U}_\xi(0,t)
    \quad\mbox{ for all } t\in(0,\infty).
  \end{align}
As $\ul{U}_\xi(0,t)$ grows at least exponentially because we have from the definition of $b$ in \eqref{defab}  and the inequality $b(t)\le b_0<\xi_0^2$ 
that
  \[
    \ul{U}_\xi(0,t)
    =\lim_{\xi\searrow0} \frac{\ul{U}( \xi,t)}{\xi} 
    = \frac{M}{\omega_n}\cdot\frac{(b(t)+\xi_0)^2}{b(t)+\xi_0^2}\cdot\frac{1}{b(t)}
    \ge \frac{M}{\omega_n}\cdot
          \frac{\xi_0^2}{2\xi_0^2}\cdot\frac{1}{b_0}\ure^{\alpha t}
    = \frac{M}{2 \omega_n b_0}\ure^{\alpha t}
  \]
for all $t\in(0,\infty)$, \eqref{goal} implies that also $u(0, \cdot)$ grows at least exponentially.
In order to obtain \eqref{goal}, we show that 
  \begin{align}\label{UulU}
    U(\xi,t) \ge \ul{U}(\xi,t)
    \quad\mbox{ for all } \xi\in[0,1]  
    \mbox{ and } t\in[0,\infty). 
  \end{align}
Thanks to  Lemma~\ref{lem4.8}, we find that \eqref{condiw0} holds. 
Therefore it follows from 
Lemmata~\ref{lem4.4}, \ref{lem4.6} and \ref{lem4.7} that 
  \begin{align}\label{com1}
    \mc{P}\ul{U}(\xi,t)\le0 
    \quad\mbox{for all}\ \xi\in(0,1)\setminus\{\xi_0\}\ 
    \mbox{and}\ t\in(0,\infty),
  \end{align}
Also, we can immediately observe 
from \eqref{defU} and \eqref{ulU} that 
$\ul{U}(0,t)=U(0,t)=0$ and $\ul{U}(1,t)=U(1,t)=\frac{M}{\omega_n}$ 
for all $t\in(0,\infty)$.
In light of these identities and \eqref{ulU0U0} as well as \eqref{com1}, 
we apply the comparison principle in Lemma~\ref{lem4.2} to 
derive \eqref{UulU}. 
Thus we can see that 
  \[
    u(0,t)=nU_\xi(0,t) 
    = \lim_{\xi\searrow0} \frac{nU( \xi,t)}{\xi} 
    \ge \lim_{\xi\searrow0} \frac{n\ul{U}( \xi,t)}{\xi}
    =n\ul{U}_\xi(0,t)
    \quad\mbox{ for all } t\in(0,\infty), 
  \]
which shows \eqref{goal} and hence concludes the proof.
\end{proof}

\section*{Acknowledgments}
The third author is supported by JSPS KAKENHI Grant Number JP22J11193.

\end{document}